\documentclass[11pt]{article}
\usepackage[margin=1in]{geometry} 
\usepackage{amssymb,amsfonts,amsmath,bbm,mathrsfs,stmaryrd}
\usepackage{xcolor}
\usepackage{url}

\usepackage[colorlinks,
             linkcolor=black!75!red,
             citecolor=blue,
             pdftitle={},
             pdfproducer={pdfLaTeX},
             pdfpagemode=None,
             bookmarksopen=true
             bookmarksnumbered=true]{hyperref}

\usepackage{tikz}
\usetikzlibrary{arrows,calc,decorations.pathreplacing,decorations.markings,intersections,shapes.geometric,through,fit,shapes.symbols,positioning,decorations.pathmorphing}

\usepackage{braket}

\usepackage[amsmath,thmmarks,hyperref]{ntheorem}
\usepackage{cleveref}

\crefname{section}{Section}{Sections}
\crefformat{section}{#2Section~#1#3} 
\Crefformat{section}{#2Section~#1#3} 

\crefname{subsection}{\S}{\S\S}
\crefformat{subsection}{#2\S#1#3} 
\Crefformat{subsection}{#2\S#1#3} 

\theoremstyle{plain}

\newtheorem{lemma}{Lemma}[section]
\newtheorem{proposition}[lemma]{Proposition}
\newtheorem{corollary}[lemma]{Corollary}
\newtheorem{theorem}[lemma]{Theorem}

\theoremstyle{nonumberplain}

\theoremstyle{plain}
\theorembodyfont{\upshape}
\theoremsymbol{\ensuremath{\blacklozenge}}

\newtheorem{definition}[lemma]{Definition}
\newtheorem{example}[lemma]{Example}
\newtheorem{remark}[lemma]{Remark}

\crefname{definition}{definition}{definitions}
\crefformat{definition}{#2definition~#1#3} 
\Crefformat{definition}{#2Definition~#1#3} 

\crefname{ex}{example}{examples}
\crefformat{example}{#2example~#1#3} 
\Crefformat{example}{#2Example~#1#3} 

\crefname{remark}{remark}{remarks}
\crefformat{remark}{#2remark~#1#3} 
\Crefformat{remark}{#2Remark~#1#3} 

\crefname{convention}{convention}{conventions}
\crefformat{convention}{#2convention~#1#3} 
\Crefformat{convention}{#2Convention~#1#3}

\crefname{lemma}{lemma}{lemmas}
\crefformat{lemma}{#2lemma~#1#3} 
\Crefformat{lemma}{#2Lemma~#1#3} 

\crefname{proposition}{proposition}{propositions}
\crefformat{proposition}{#2proposition~#1#3} 
\Crefformat{proposition}{#2Proposition~#1#3} 

\crefname{corollary}{corollary}{corollaries}
\crefformat{corollary}{#2corollary~#1#3} 
\Crefformat{corollary}{#2Corollary~#1#3} 

\crefname{theorem}{theorem}{theorems}
\crefformat{theorem}{#2theorem~#1#3} 
\Crefformat{theorem}{#2Theorem~#1#3} 

\crefname{assumption}{assumption}{Assumptions}
\crefformat{assumption}{#2assumption~#1#3} 
\Crefformat{assumption}{#2Assumption~#1#3} 

\crefname{equation}{}{}
\crefformat{equation}{(#2#1#3)} 
\Crefformat{equation}{(#2#1#3)}

\theoremstyle{nonumberplain}
\theoremsymbol{\ensuremath{\blacksquare}}

\newtheorem{proof}{Proof}

\newcommand\bC{{\mathbb C}}

\newcommand\bS{{\mathbb S}}

\newcommand\bZ{{\mathbb Z}}

\newcommand\cO{{\mathcal O}}
\newcommand\cP{{\mathcal P}}

\DeclareMathOperator{\id}{id}

\newcommand{\qedhere}{\mbox{}\hfill\ensuremath{\blacksquare}}

\title{Leavitt vs. $C^*$ pullbacks}
\author{Alexandru Chirvasitu}

\begin{document}

\date{}

\newcommand{\Addresses}{{
  \bigskip
  \footnotesize

  \textsc{Department of Mathematics, University at Buffalo, Buffalo,
    NY 14260-2900, USA}\par\nopagebreak \textit{E-mail address}:
  \texttt{achirvas@buffalo.edu}

}}

\maketitle

\begin{abstract}
  We show that certain pullbacks of $*$-algebras equivariant with respect to a compact group action remain pullbacks upon completing to $C^*$-algebras. This unifies a number of results in the literature on graph algebras, showing that pullbacks of Leavitt path algebras lift automatically to pullbacks of the corresponding graph $C^*$-algebras.
\end{abstract}


\vspace{.5cm}



\section*{Introduction}

The present note was prompted by \cite{trm-zqi,trm-fXl} and the desire to explain the relationship between their respective main results, \cite[Theorem, p.2]{trm-zqi} and \cite[Theorem 3.2]{trm-fXl} respectively. The former puts a pullback structure on a graph $C^*$-algebra, whereas the latter proves the parallel result for the corresponding Leavitt path algebras.

Given the close relationship between the two settings, it would be natural to seek an abstract framework that would allow one to simply morph the Leavitt result into its $C^*$ analogue without having to retrace the proofs. We propose such a framework here. 

\subsection*{Acknowledgements}

The author was partially supported by NSF grant DMS-1801011. The work is also part of the project QUANTUM DYNAMICS supported by EU grant H2020-MSCA-RISE-2015-691246 and Polish Government grant 3542/H2020/2016/2.

Exchanges with Piotr M. Hajac and Mariusz Tobolski have been very enlightening.

\section{Preliminaries}\label{se.prel}

\subsection{Generalities on equivariant structures}\label{subse.gcast}

We will work with {\it $G$-equivariant structures} on $*$ or $C^*$-algebras (or more generally Banach spaces) for a compact group $G$. For Banach spaces $V$ this simply means a {\it strongly continuous} action as automorphisms (isometries for Banach spaces or $C^*$ automorphisms for $C^*$-algebras) in the sense that
\begin{equation*}
G\ni g\mapsto gv \in V 
\end{equation*}
is a continuous map for each $v\in V$; this is the customary and apparently most appropriate continuity assumption on actions on operator algebras (e.g. \cite[\S 2.2]{phil-equiv}).

On the other hand, for $*$-algebras $A$ a $G$-equivariant structure means a comodule structure
\begin{equation*}
  A\to A\otimes \cO(G),
\end{equation*}
where $\cO(G)$ is the Hopf algebra of {\it representative functions} on $G$: the span of matrix coefficients of finite-dimensional continuous $G$-representations. When $G=\bS^1$, which will be the case in the concrete applications discussed below, $\cO(G)$ is nothing but the algebra $\bC[t^{\pm 1}]$ of Laurent polynomials and an equivariant structure is a $\bZ$-grading on the $*$-algebra $A$. 

For a $G$-$C^*$-algebra $A$ we have the faithful expectation
\begin{equation*}
  A\ni a\mapsto \int_G g\triangleright a \ \mathrm{d}\mu(g)\in A^G
\end{equation*}
onto the fixed-point subalgebra $A^G\le A$, where $\mathrm{d}\mu$ is the Haar measure on $G$.

Given a $G$-equivariant inclusion $A\le \overline{A}$ of a $*$-algebra into a $C^*$-algebra we have (see e.g. \cite[Proposition 6.1]{cpt})

\begin{lemma}
  The invariant subalgebra $A^G$ is dense in $\overline{A}^G$. 
\qedhere  
\end{lemma}

A number of elementary remarks apply to the functor $A\mapsto A^G$ from $G$-$C^*$-algebras to $C^*$-algebras. First, it is a right adjoint with left adjoint
\begin{equation*}
A\mapsto A\text{ equipped with the trivial }G-\text{action}.
\end{equation*}
In particular the fixed-point functor is continuous (i.e. it preserves all limits) \cite[\S V.5, Theorem 1]{mcl}. Additionally:

\begin{lemma}\label{le.monic}
  The fixed-point functor $A\mapsto A^G$ on $G$-$C^*$-algebras reflects monomorphisms, in the sense that if $f:A\to B$ is a morphism of $G$-$C^*$-algebras such that
  \begin{equation*}
    f^G:A^G\to B^G
  \end{equation*}
  is monic then so is $f$.
\end{lemma}
\begin{proof}
  Indeed, since the expectation $A\to A^G$ is faithful, the kernel $J$ of $A\to B$ vanishes if and only if $J^G$ does, which happens by hypothesis.
\end{proof}


\subsection{Algebras attached to graphs}\label{subse.grph}

We will need some background on graph algebras: Leavitt path algebras as in \cite[Definition 1.2.3]{aasm-leavitt} and their analytic counterparts, graph $C^*$-algebras, defined, say, as in \cite[\S 2]{dt}. Briefly:

\begin{definition}\label{def.grph}
  A {\it graph} is a quadruple $E=(E^0,E^1,s,t)$ consisting of sets $E^0$ and $E^1$ of {\it vertices} and respectively {\it edges} and {\it source} and {\it target} maps $s,t:E^1\to E^0$.

  A vertex $v\in E^0$ is {\it regular} if $s^{-1}(v)$ is finite and non-empty (we also say that $v$ is a {\it finite emitter} and is not a {\it sink}).
\end{definition}

With this in place, recall \cite[Definition 1.2.3]{aasm-leavitt}:

\begin{definition}\label{def.lpa}
  Given a graph $E$, the {\it Leavitt path algebra} $L_k(E)$ over a unital commutative ring $k$ is the $k$-algebra equipped with an anti-multiplicative involution `$*$' defined by generators $v=v^*$ for $v\in E^0$ and $e,e^*$ for $e\in E^1$ subject to relations
  \begin{itemize}
  \item $vv'=\delta_{v,v'}$;
  \item $s(e)e=er(e)=e$;
  \item $e^*e'=\delta_{e,e'}r(e)$;
  \item $\sum_{e\in s^{-1}(v)}ee^*=v$ for all regular vertices $v\in E^0$.    
  \end{itemize}

  While in this generality the involution `$*$' is assumed to act as the identity on $k$, for $k=\bC$ one usually works with a modified definition whereby `$*$' is complex conjugation on $\bC$. This makes the corresponding algebra (which we will then denote simply by $L(E)$) into a complex $*$-algebra. 
\end{definition}

On the other hand (see \cite[\S 2]{dt} or \cite[\S 5.2]{aasm-leavitt}):

\begin{definition}\label{def.gca}
  The {\it graph $C^*$-algebra} $C^*(E)$ is the universal $C^*$-algebra generated by $v=v^*$, $v\in E^0$ and $e,e^*$ for $e\in E^1$ subject to the same relations as in \Cref{def.lpa}, together with the additional conditions
  \begin{equation*}
    ee^*\le s(e),\ \forall e\in E^1.
  \end{equation*}
\end{definition}

The additional requirement in \Cref{def.gca} is not needed if the graph $E$ is {\it row-finite}, i.e. if every vertex emits finitely many (possibly zero) edges. All graphs considered in \cite{trm-fXl,trm-zqi,hrt} referred to below are row-finite.

\begin{remark}
  It turns out that $C^*(E)$ is nothing but the $C^*$ envelope of $L(E)$, i.e. the obvious map $L(E)\to C^*(E)$ is an initial object in the category of $*$-morphisms from $L(E)$ into $C^*$-algebras \cite[\S 5.2]{aasm-leavitt}. 
\end{remark}

Graph algebras (Leavitt of $C^*$) admit $\bS^1$-actions in this sense (called the {\it gauge actions} in the literature \cite[Chapter 2]{raeb-bk} or \cite[\S 2.1]{aasm-leavitt}): $z\in \bS^1$ simply scales every edge generator $e\in E^1$ by $z$ in the $C^*$ case, while the corresponding grading on the corresponding Leavitt path algebra assigns degree $1$ to each $e$ and degree $-1$ to each $e^*$.

\section{Graph algebra pullbacks}\label{se.main}

Let $A\le \cP_G(\overline{A})\le \overline{A}$ be a dense inclusion of a $G$-$*$-algebra into a $G$-$C^*$-algebra as in \Cref{se.prel}. Our first remark is

\begin{lemma}\label{le.afag}
  If the invariant subalgebra $A^G$ admits a unique $C^*$ norm then the norm of $\overline{A}$ is the unique $G$-invariant $C^*$-norm on $A$.
\end{lemma}
\begin{proof}
  Consider a surjection $f:\overline{A}\to B$ of $G$-$C^*$-algebras, faithful on $A$. Its restriction $\overline{A}^G\to B^G$ is monic by the unique-norm assumption and hence $f$ is monic by \Cref{le.monic}. Since $f$ was also onto, we are done.
\end{proof}

\begin{corollary}\label{cor.unq}
  If the invariant subalgebra $A^G$ is AF then the norm of $\overline{A}$ is the only $G$-invariant $C^*$-norm on $A$.  
\end{corollary}
\begin{proof}
Indeed, AF $*$-subalgebras of $C^*$-algebras have unique $C^*$-norms, so \Cref{le.afag} applies. 
\end{proof}

We consider a diagram
\begin{equation}\label{eq:6}
  \begin{tikzpicture}[auto,baseline=(current  bounding  box.center)]
    \path[anchor=base] (0,0) node (a) {$A$} +(3,-.5) node (d) {$D$} +(6,0) node (b) {$B$} +(3,.5) node (c) {$C$};
    \draw[->] (a) to[bend right=6] node[pos=.5,auto,swap] {$\scriptstyle \ell$} (d);
    \draw[->] (b) to[bend left=6] node[pos=.5,auto] {$\scriptstyle r$} (d);
    \draw[->] (c) to[bend right=6] node[pos=.5,auto,swap] {$$} (a);
    \draw[->] (c) to[bend left=6] node[pos=.5,auto,swap] {$$} (b);
  \end{tikzpicture}  
\end{equation}
in the category of $G$-$*$-algebras.

To state the main result of this section we need

\begin{lemma}\label{le:ker-is-dense}
Let $\pi:M\to N$ be a $*$-algebra surjection embedding in its $C^*$ completion $\overline{\pi}:\overline{M}\to \overline{M}$. Then, $\ker \pi$ is dense in $\ker\overline{\pi}$.   
\end{lemma}
\begin{proof}
  First, note that the closure of $J:=\ker \pi$ in $\overline{M}$ is a $C^*$ ideal contained in $\ker\overline{\pi}$, so we have a factorization
  \begin{equation*}
 \begin{tikzpicture}[auto,baseline=(current  bounding  box.center)]
   \path[anchor=base] (0,0) node (1) {$\overline{M}$}
   +(2,.5) node (2) {$\overline{M}/\overline{J}$}
   +(4,0) node (3) {$\overline{N}$};
   \draw[->] (1) to[bend left=6] node[pos=.5,auto] {$\scriptstyle $} (2);
   \draw[->] (2) to[bend left=6] node[pos=.5,auto] {$\scriptstyle \bullet$} (3);
   \draw[->] (1) to[bend right=6] node[pos=.5,auto,swap] {$\scriptstyle \overline{\pi}$} (3);
 \end{tikzpicture}
\end{equation*}
The fact that the $\bullet$ map in the above diagram is an isomorphism follows from an examination of the (Hilbert space) representations of the two $C^*$-algebras: those of $\overline{M}/\overline{J}$ are precisely the $*$-representations of $M$ vanishing on $J=\ker \pi$, and hence precisely the $*$-representations of $N$ (and hence of $\overline{M}$).
\end{proof}


As a consequence, we have the following simple remark. 

\begin{lemma}\label{le:adm-surj}
  Let $\pi:M\to N$ be a $*$-algebra surjection embedding into its $C^*$ completion $\overline{\pi}$. Then, for every $y\in N$ of norm $\|y\|<\varepsilon$ in the completion $\overline{N}$ there is
  \begin{equation*}
    x\in \pi^{-1}(y),\ \|x\|<\varepsilon\text{ in }\overline{M}. 
  \end{equation*}
\end{lemma}
\begin{proof}
  On the one hand we can find
  \begin{equation*}
    \overline{x}\in (\overline{\pi})^{-1}(y)\subset \overline{M}
  \end{equation*}
  of norm $<\varepsilon$ simply because $\overline{\pi}$ is a $C^*$-algebra surjection. On the other hand, the surjectivity of $\pi$ ensures the existence of {\it some} $x'\in \pi^{-1}(y)\subset M$, perhaps not satisfying the norm constraint. But then $\overline{x}-x'\in \ker\overline{\pi}$, meaning by \Cref{le:ker-is-dense} that it is arbitrarily approximable with elements $x-x'\in \ker\pi$. If the latter is sufficiently close to $\overline{x}-x'$ then $x\in A$ will be sufficiently close to $\overline{x}$ to achieve $\|x\|<\varepsilon$.
\end{proof}

We are now ready for the main statement.

\begin{theorem}\label{th.cmpl}
  Assume that
  \begin{itemize}
  \item \Cref{eq:6} is a pullback;
  \item $r$ is onto;
  \item $C^G$ is AF;  
  \item \Cref{eq:6} embeds in its $C^*$ completion.  
  \end{itemize}
  Then, the $C^*$-completed diagram
  \begin{equation}\label{eq:7}
  \begin{tikzpicture}[auto,baseline=(current  bounding  box.center)]
    \path[anchor=base] (0,0) node (a) {$\overline A$} +(3,-.5) node (d) {$\overline D$} +(6,0) node (b) {$\overline B$} +(3,.5) node (c) {$\overline C$};
    \draw[->] (a) to[bend right=6] node[pos=.5,auto,swap] {$\scriptstyle \ell$} (d);
    \draw[->] (b) to[bend left=6] node[pos=.5,auto] {$\scriptstyle r$} (d);
    \draw[->] (c) to[bend right=6] node[pos=.5,auto,swap] {$$} (a);
    \draw[->] (c) to[bend left=6] node[pos=.5,auto,swap] {$$} (b);
  \end{tikzpicture}  
\end{equation}
is a pullback of $G$-$C^*$-algebras. 
\end{theorem}
\begin{proof}
Consider the $C^*$-pullback
  \begin{equation}\label{eq:1}
  \begin{tikzpicture}[auto,baseline=(current  bounding  box.center)]
    \path[anchor=base] (0,0) node (a) {$\overline A$} +(3,-.5) node (d) {$\overline D$} +(6,0) node (b) {$\overline B$.} +(3,.5) node (c) {$P$};
    \draw[->] (a) to[bend right=6] node[pos=.5,auto,swap] {$\scriptstyle \ell$} (d);
    \draw[->] (b) to[bend left=6] node[pos=.5,auto] {$\scriptstyle r$} (d);
    \draw[->] (c) to[bend right=6] node[pos=.5,auto,swap] {$$} (a);
    \draw[->] (c) to[bend left=6] node[pos=.5,auto,swap] {$$} (b);
  \end{tikzpicture}  
\end{equation}
By \Cref{cor.unq} the canonical map $\overline{C}\to P$ is one-to-one, because $C$ admits a unique $G$-invariant $C^*$ norm, which must be the universal one inherited from $\overline{C}$. 

On the other hand, the surjectivity of $\overline{C}\to P$ follows from that of $r$. To see this, consider a pair of elements
\begin{equation*}
  a\in \overline{A},\ b\in \overline{B}
\end{equation*}
with equal images in $\overline{D}$ through $\ell$ and $r$ in \Cref{eq:1} respectively. The pair forms an element $(a,b)\in P$, by the very definition of a pullback. We want to argue that $(a,b)$ is arbitrarily approximable in the sup norm on $A\times B$ by elements $(a_i,b_i)$ in the pullback $C$, i.e. such that $\ell(a_i)=r(b_i)$.

First, approximate $(a,b)$ arbitrarily well with $(a_i,b_i')\in A\times B$ disregarding for the moment the coincidence of their images through $\ell$ and $r$ respectively. We have
\begin{equation*}
  \|\ell(a_i)-r(b_i')\|<\varepsilon.
\end{equation*}
\Cref{le:adm-surj} ensures the existence of some $\gamma_i\in B$ with
\begin{itemize}
\item $\|\gamma_i\|<\varepsilon$
\item $r(\gamma_i)=r(b_i')-\ell(a_i)$. 
\end{itemize}
Finally, set $b_i=b_i'-\gamma_i$. This element will
\begin{itemize}
\item be close to the original element $b$;
\item have image $\ell(a_i)$ through $r$.  
\end{itemize}
In short, $(a_i,b_i)\in C$ will be a good approximation for $(a,b)\in P$.
\end{proof}

The surjectivity assumption is essential in \Cref{th.cmpl}, as the following remark shows.

\begin{example}\label{ex:must-surj}
  In \Cref{eq:6}, we take $D=C(\bS^1)$, which will hence be its own $C^*$ closure. $A$ and $B$ will be dense $*$-subalgebras thereof such that
  \begin{itemize}
  \item $\overline{A}=\overline{B}=D$;
  \item $A\cap B$ consists of precisely the scalars. 
  \end{itemize}
  To achieve all of this, first let $A$ be the $*$-algebra generated by the standard unitary
  \begin{equation*}
    \id:\bS^1\to \bS^1\subset \bC.
  \end{equation*}
  We indeed have $\overline{A}=C$, sine the latter is the universal $C^*$-algebra generated by a unitary.

  On the other hand, for $B$ we take the ``twist'' of $A$ by a sufficiently general self-homeomorphism $\varphi$ of the circle $\bS^1$; in other words,
  \begin{equation*}
    B = \{f\circ\varphi\ |\ f\in A\}. 
  \end{equation*}
  ``Sufficiently general'' here might mean, for instance, that the set
  \begin{equation*}
    \{z\in \bS^1\ |\ \varphi\text{ is not differentiable at }z\}
  \end{equation*}
  is dense in the unit circle. To see that this will ensure $A\cap B=\bC$, note that every non-constant trigonometric polynomial $f\in A$ will map some interval arc $J\subseteq \bS^1$ diffeomorphically onto its image. But then, denoting by
  \begin{equation*}
    \psi:f(J)\to J
  \end{equation*}
  the inverse of that diffeomorphism, $\psi\circ f\circ\varphi|_J$ equals $\varphi$ on $J$. Since the latter has points of non-differentiability in $J$, the non-constant $f\circ\varphi\in B$ cannot possibly be equal to any element of $A$.
\end{example}

We now apply the preceding discussion to the setup of \cite{trm-zqi,trm-fXl}. Recall that the authors of said papers work with graphs $Q'\subset Q''$ and maps
\begin{equation}\label{eq:4}
  \begin{tikzpicture}[auto,baseline=(current  bounding  box.center)]
    \path[anchor=base] (0,0) node (u) {$L(Q')$} +(4,-.5) node (w) {$L(Q')\otimes k[t,t^{-1}]$} +(8,0) node (v) {$L(Q'')\otimes k[t,t^{-1}]$};
    \draw[->] (u) to[bend right=6] node[pos=.5,auto,swap] {$\scriptstyle \delta$} (w);
    \draw[->] (v) to[bend left=6] node[pos=.5,auto] {$\scriptstyle \pi\otimes\id$} (w);
  \end{tikzpicture}   
\end{equation}
where $L(-)$ denotes the Leavitt path algebra construction, $k$ denotes a ground field, and $\pi$ is onto. There is an analogous picture for graph $C^*$-algebras, whereupon $k=\bC$; this is the case of interest here.

The two papers work in the Leavitt path algebra and $C^*$ setting (\cite{trm-fXl,trm-zqi} respectively), proving parallel results: while \cite[Theorem 3.2]{trm-fXl} argues that a certain diagram
\begin{equation}\label{eq:10}
  \begin{tikzpicture}[auto,baseline=(current  bounding  box.center)]
    \path[anchor=base] (0,0) node (u) {$L(Q')$} +(4,-.5) node (w) {$L(Q')\otimes k[t,t^{-1}]$} +(8,0) node (v) {$L(Q'')\otimes k[t,t^{-1}]$}  +(4,.5) node (top) {$L(Q)$};
    \draw[->] (u) to[bend right=6] node[pos=.5,auto,swap] {$\scriptstyle \delta$} (w);
    \draw[->] (v) to[bend left=6] node[pos=.5,auto] {$\scriptstyle \pi\otimes\id$} (w);
    \draw[->] (top) to[bend right=6] node[pos=.5,auto,swap] {$\scriptstyle $} (u);
    \draw[->] (top) to[bend left=6] node[pos=.5,auto,swap] {$\scriptstyle $} (v);
  \end{tikzpicture}   
\end{equation}
is a pullback of graded $*$-algebras, the analytic analogue \cite[Theorem, p.2]{trm-zqi} shows that the $C^*$ graph algebra version
\begin{equation}\label{eq:11}
  \begin{tikzpicture}[auto,baseline=(current  bounding  box.center)]
    \path[anchor=base] (0,0) node (u) {$C^*(Q')$} +(4,-.5) node (w) {$C^*(Q')\otimes C(\bS^1)$} +(8,0) node (v) {$C^*(Q'')\otimes C(\bS^1)$}  +(4,.5) node (top) {$C^*(Q)$};
    \draw[->] (u) to[bend right=6] node[pos=.5,auto,swap] {$\scriptstyle \delta$} (w);
    \draw[->] (v) to[bend left=6] node[pos=.5,auto] {$\scriptstyle \pi\otimes\id$} (w);
    \draw[->] (top) to[bend right=6] node[pos=.5,auto,swap] {$\scriptstyle $} (u);
    \draw[->] (top) to[bend left=6] node[pos=.5,auto,swap] {$\scriptstyle $} (v);
  \end{tikzpicture}   
\end{equation}
is a pullback of $C^*$-algebras equipped with actions by the circle group $\bS^1$ (with $C(-)$ denoting the algebra of continuous functions).

We now have:

\begin{proposition}\label{pr.1st}
  \cite[Theorem 3.2]{trm-fXl} implies \cite[Theorem, p.2]{trm-zqi}. 
\end{proposition}
\begin{proof}
  This follows from \Cref{th.cmpl} for $G=\bS^1$ and the fact that the $\bS^1$-invariant subalgebras of Leavitt path algebras are AF.
\end{proof}
  

The pattern recurs in \cite{hrt}: a pushout graph diagram
\begin{equation*}\label{eq:12}
  \begin{tikzpicture}[auto,baseline=(current  bounding  box.center)]
    \path[anchor=base] (0,0) node (a) {$F_1$} +(3,-.5) node (d) {$F_1\cap F_2$} +(6,0) node (b) {$F_2$} +(3,.5) node (c) {$E$};
    \draw[->] (d) to[bend left=6] node[pos=.5,auto,swap] {$\scriptstyle $} (a);
    \draw[->] (d) to[bend right=6] node[pos=.5,auto] {$\scriptstyle $} (b);
    \draw[->] (a) to[bend left=6] node[pos=.5,auto,swap] {$$} (c);
    \draw[->] (b) to[bend right=6] node[pos=.5,auto,swap] {$$} (c);
  \end{tikzpicture}  
\end{equation*}
results via \cite[Theorem 3.1]{hrt} in a pullback diagram in the category of $\bS^1$-$C^*$-algebras
\begin{equation}\label{eq:13}
  \begin{tikzpicture}[auto,baseline=(current  bounding  box.center)]
    \path[anchor=base] (0,0) node (u) {$C^*(F_1)$} +(4,-.5) node (w) {$C^*(F_1\cap F_2)$} +(8,0) node (v) {$C^*(F_2)$}  +(4,.5) node (top) {$C^*(E)$};
    \draw[->] (u) to[bend right=6] node[pos=.5,auto,swap] {$\scriptstyle $} (w);
    \draw[->] (v) to[bend left=6] node[pos=.5,auto] {$\scriptstyle $} (w);
    \draw[->] (top) to[bend right=6] node[pos=.5,auto,swap] {$\scriptstyle $} (u);
    \draw[->] (top) to[bend left=6] node[pos=.5,auto,swap] {$\scriptstyle $} (v);
  \end{tikzpicture}   
\end{equation}
consisting of surjections only. It is then observed in \cite[Remark 3.2]{hrt} that a parallel proof would dispatch the Leavitt path algebra version (with all instances of $C^*(-)$ replaced by the corresponding $L(-)$). In that context, the analogue of \Cref{pr.1st} is 

\begin{proposition}
  \cite[Theorem 3.1]{hrt} follows from its Leavitt path algebra analogue.
  \qedhere
\end{proposition}


On the other hand, the results of \cite{cht} involve {\it non}-equivariant pullbacks of graph algebras. The general principle at work is the same however:

\begin{proposition}
  \cite[Theorem 3.3]{cht} follows from its Leavitt path algebra analogue. 
\end{proposition}
\begin{proof}
  This time around the compact group $G$ will be trivial. The graph algebra at the top of the diagram, corresponding to $C$ in \Cref{eq:6}, is attached to a graph assumed to have no loops. It follows that the Leavitt path algebra is AF (e.g. \cite[Theorem 2.4]{kpr}) and the conclusion once more follows from \Cref{th.cmpl}.
\end{proof}



\bibliographystyle{plain}
\addcontentsline{toc}{section}{References}

\def\polhk#1{\setbox0=\hbox{#1}{\ooalign{\hidewidth
  \lower1.5ex\hbox{`}\hidewidth\crcr\unhbox0}}}

\Addresses

\end{document}